\documentclass[12pt]{amsart}

\usepackage{amssymb,color}

\numberwithin{equation}{section}

\newtheorem{theorem}[subsection]{Theorem}
\newtheorem{lemma}[subsection]{Lemma}
\newtheorem{Prop}[subsection]{Proposition}

\newtheorem{conjecture}[subsection]{Conjecture}
\newtheorem{remark}[subsection]{Remark}

\theoremstyle{definition}

\theoremstyle{remark}

\newcommand{\bz}{{\bf{Z}}}

\newcommand{\zp}{{{\bf{Z}}/p}}

\newcommand{\fp}{{\bf{F}}_p}

\newcommand{\field}{{\bf F}}
\newcommand{\rfield}{{\bf k}}

\newcommand{\og}{{\overline{\gamma}}}
\newcommand{\of}{{\overline{f}}}

\newcommand{\rightquot}{/\!\!/}

\newcommand{\sltwo}{SL_2(\field_p)}

\newcommand{\lt}{\mathop{\rm LT}}
\newcommand{\lm}{\mathop{\rm LM}}

\newcommand{\tat}{t\^ete-\`a-t\^ete}

\title[Elementary abelian $p$-groups]
{Rings of invariants for modular representations
of elementary abelian $p$-groups}

\author{H.E.A. Campbell}
\address{Department of Mathematics\\
\hfil\break\indent University of New Brunswick, Fredricton, NB, E3B 5A3, Canada }
\email{heac@unb.ca}

\author{R.J. Shank}
\address{School of Mathematics, Statistics \&  Actuarial Science \\
\hfil\break\indent University of Kent, Canterbury, CT2 7NF, UK}
\email{R.J.Shank@kent.ac.uk}

\author{D.L. Wehlau}
\address{Department of Mathematics and Computer Science \\
\hfil\break\indent Royal Military College, Kingston, ON, K7K 7B4, Canada}
\email{wehlau-d@rmc.ca}

\subjclass{13A50}
\date{\today}

\begin{document}

\begin{abstract}
We initiate a study of the rings of invariants of modular representations
of elementary abelian $p$-groups. 
With a few notable exceptions, the modular representation theory of an elementary abelian $p$-group
is wild. However, for a given dimension, it is possible to parameterise the
representations. We describe parameterisations for modular representations of dimension two and
of dimension three. We compute the ring of invariants for all two dimensional representations; these
rings are generated by two algebraically independent elements. 
We compute the ring of invariants of the symmetric square of a two dimensional representation;
these rings are hypersurfaces. 
We compute the ring of invariants
for all three dimensional representations of rank at most three; these rings are complete intersections with 
embedding dimension at most five. We conjecture that the ring of invariants for any 
three dimensional representation of an elementary abelian $p$-group  is a complete intersection. 
 \end{abstract}

\maketitle

\section{Introduction}
We initiate a study of the rings of invariants of modular representations
of elementary abelian $p$-groups. 
With a few notable exceptions, the modular representation theory of an elementary abelian $p$-group
is wild, see for example \cite[\S 4.4]{benson_rep1}. However, for a given dimension, it is possible to parameterise the
representations. We describe parameterisations for modular representations of dimension two and
of dimension three. We compute the ring of invariants for all two dimensional representations; these
rings are generated by two algebraically independent elements. 
We compute the ring of invariants of the symmetric square of a two dimensional representation;
these rings are hypersurfaces. 
We compute the ring of invariants
for all three dimensional representations of rank at most three; these rings are complete intersections with 
embedding dimension at most five. We conjecture that the ring of invariants for a 
three dimensional representation of rank $r$ is a complete intersection with embedding dimension
at most $\lceil r/2 \rceil +3$.
 
Let $V$ denote an $n$ dimensional representation of a group $G$, over a field $\field$ of characteristic $p$,
for a prime number $p$. We will usually assume that $G$ is finite and that $p$ divides the order of $G$,
in other words, that $V$ is a modular representation of $G$.
We view $V$ as a left module over the group ring $\field G$ and the dual, $V^*$,
as a right $\field G$-module. Let $\field[V]$ denote the symmetric algebra on $V^*$.
The action of $G$ on $V^*$ extends to an action by 
degree preserving algebra automorphisms  on $\field[V]$.
Choosing a basis $\{x_1,x_2,\ldots,x_n\}$ allows us to identify $\field[V]$ with the algebra
of polynomials $\field[x_1,x_2,\ldots,x_n]$.
Our convention that $\field[V]$ is a right $\field G$-module is consistent
with the convention used by the invariant theory package in the computer algebra 
software Magma~\cite{magma}. The ring of invariants, $\field[V]^G$, is the subring
of $\field[V]$ consisting of those polynomials fixed by the action of $G$.
Note that elements of $\field[V]$ represent polynomial functions on $V$ and
that elements of $\field[V]^G$ represent polynomial functions on the set of orbits
$V/G$. For $G$ finite and $\field$ algebraically closed, $\field[V]^G$ is the
ring of regular functions on the categorical quotient $V\rightquot G$.
For background material on the invariant theory of finite groups, see \cite{Benson},
\cite{CW}, \cite{DK}, or \cite{NS}.

Choosing a basis for $V$ (or $V^*$) determines a group homomorphism $\rho:G \to GL_n(\field)$.
Conversely, a group homomorphism $\rho$ can be used to define a right $\field G$-module structure
on $\field^n$.  Our philosophy is to use the set of group homomorphism 
 ${\rm hom}(G,GL_n(\field))$  to parameterise representations of $G$ over $\field$ of dimension $n$.
From this point of view, every group homomorphism determines a subring 
$\field[x_1,\ldots,x_n]^{\rho(G)}\subseteq \field[x_1,\ldots,x_n]$.
 Conjugation can be used to define a left action of $GL_n(\field)$ on ${\rm hom}(G,GL_n(\field))$  and
the left action of the automorphism group ${\rm Aut}(G)$ on $G$ induces a right action of ${\rm Aut}(G)$ on
${\rm hom}(G,GL_n(\field))$. Equivalent representations give conjugate group homomorphisms and
the corresponding subrings are isomorphic but not necessarily equal. The action of
${\rm Aut}(G)$ preserves the image of the group homomorphism and, therefore,  automorphisms
in the same ${\rm Aut}(G)$-orbit determine the same subring.
Our goal is to compute the subrings of $\field[x_1,\ldots,x_n]$ corresponding to the $GL_n(\field)$-orbits
of ${\rm hom}(G,GL_n(\field))\rightquot {\rm Aut}(G)$. In practice, we consider a variety 
$\mathcal{V}\subset {\rm hom}(G,GL_n(\field))$ and a subgroup $H< GL_n(\field)$, which acts on 
$\mathcal{V}$ by conjugation, and compute the subrings corresponding to the $H$-orbits of 
$\mathcal{V}\rightquot {\rm Aut}(G)$. If $G=(\zp)^r=\langle e_1,\ldots,e_r\rangle$ is an elementary abelian
$p$-group, then ${\rm Aut}(G)\cong GL_r(\fp)$.

We make extensive use of the theory of SAGBI bases to compute rings of invariants. 
A SAGBI basis is the {\bf S}ubalgebra {\bf A}nalogue of a {\bf G}r\"obner {\bf B}asis for {\bf I}deals.
The concept was introduced independently by Robbiano-Sweedler~\cite{rs} and Kapur-Madlener~\cite{km};
a useful reference is Chapter~11 of Sturmfels~\cite{Sturmfels}. 
We adopt the convention that a monomial
is a product of variables and a term is a monomial with a coefficient. For a polynomial $f\in\field[x_1,\ldots,x_n]$,
we denote the lead monomial of $f$ by $\lm(f)$ and the lead term of $f$ by $\lt(f)$.
For $\mathcal{B}=\{f_1,\ldots,f_{\ell}\}\subset \field[x_1,\ldots,x_n]$ and $I=(i_1,\ldots,i_{\ell})$, a sequence of
non-negative integers, denote $\prod_{j=1}^{\ell}f^{i_j}$ by $f^I$. A {\it \tat\ } for $\mathcal{B}$ is a pair
$(f^I,f^J)$ with $\lm(f^I)=\lm(f^J)$; we say that a \tat\ is {\it non-trivial} if the support of $I$ is disjoint from the support of $J$. 
The reduction of an $S$-polynomial is a fundamental calculation in the theory of Gr\"obner bases.
The analogous calculation for SAGBI bases is the {\it subduction} of a \tat.  
A subset $\mathcal{B}$ of a subalgebra $A\subset\field[x_1,\ldots,x_n]$ is a SAGBI basis for $A$
if the lead monomials of the elements of $\mathcal{B}$ generate the lead term algebra  of $A$ or, equivalently, every non-trivial \tat\ for $\mathcal{B}$ subducts to zero.   For background material on term orders and Gr\"obner bases, 
we recommend \cite{AL}.

We conclude the introduction with a summary.
In Section~\ref{sdx_sec}, we introduce the SAGBI/Divide-by-$x$ algorithm which, in principle, 
can be used to compute the ring of invariants for any modular representation of a $p$-group.
We also prove a result (Theorem~\ref{genthm}) which provides sufficient conditions to show that a 
set of invariants generates the full ring of invariants.  In the later sections we repeatedly use this result to show that
we have correctly constructed a set of generators for the ring of invariants.
Section~\ref{2drep_sec} is devoted to describing the two dimensional modular representations
of $p$-groups and computing the corresponding rings of invariants.   We observe that any $p$-group $G$ having a faithful
two dimensional modular representation must be elementary abelian and the corresponding ring of invariants is 
a polynomial algebra generated
by two algebraically independent polynomials, one of degree $1$ and the other of degree $|G|$.
An explicit description of the degree $|G|$ invariant plays
an important r\^ole in the calculation of the ring of invariants for the symmetric square
of the dimension two representation. This calculation is is  given in Section~\ref{ss_sec};
in all cases the ring of invariants is a hypersurface with $4$ generators and $1$ relation. 
In Section~\ref{cl3d_sec}, we describe the three dimensional representations of $(\zp)^r$.
Section~\ref{field_section} includes the construction of a particularly nice generating set for
the field of fractions of the ring of invariants for a generic three dimensional representation of $(\zp)^r$.
In Section~\ref{r2d3_sec}, we compute the ring of invariants for all three dimensional representations
of $(\zp)^2$ and in Section~\ref{r3d3_sec} we compute the ring of invariants for all
three dimensional representations of $(\zp)^3$.   We classify these representations first by using their fixed point
sets and then by polynomial conditions.  In all cases the ring of invariants is a complete intersection.
The final section is devoted to conclusions and conjectures.

\section{The SAGBI/Divide-by-$x$ Algorithm}\label{sdx_sec}

In this section $G$ is a $p$-group, $\field$ is an arbitrary field of characteristic $p$ and
$V$ is an $\field G$-module of dimension $n$. Choose a basis $\{x,y_1,\ldots,y_{n-1}\}$ for $V^*$
so that $x\in (V^*)^G$ and $y_iG\in{\rm Span}_{\field}\{x,y_j\mid j\leq i\}$, i.e., the action of $G$
is by upper-triangular matrices. We use the graded reverse lexicographic order
with $x<y_1<y_2<\cdots <y_{n-1}$.
Suppose that we have constructed homogeneous $f_1,\ldots,f_{\ell}\in\field[V]^G$ with
$\lm(f_i)\in\field[y_1,\ldots,y_{n-1}]$ .
Define $\mathcal{B}:=\{x,f_1,\ldots,f_{\ell}\}$ and let $A$ denote the algebra generated by $\mathcal{B}$.

\begin{theorem}\label{genthm} Suppose $A[x^{-1}]=\field[V]^G[x^{-1}]$, $\mathcal{B}$ is a SAGBI bases for
$A$ and $\field[V]^G$ is an integral extension of $A$. Then $A=\field[V]^G$ and $\mathcal{B}$
is SAGBI basis for $\field[V]^G$.
\end{theorem}
\begin{proof}
Since $A$ and $\field[V]^G$ have the same field of fractions and $\field[V]^G$ is an integral extension of $A$,
to prove $A=\field[V]^G$ it is sufficient to show that $A$ is normal, i.e., integrally closed in its field of fractions.
Unique Factorisation Domains are normal; therefore it is sufficient to show that $A$ is UFD.

It is well known that the ring of invariants for a modular representation of a $p$-group is a UFD
(see for example \cite[Corollary~3.9.3]{Benson}).
Therefore $\field[V]^G$ is a UFD. Thus $A[x^{-1}]=\field[V]^G[x^{-1}]$ is a UFD.
It follows from \cite[Theorem~20.2]{Mat}
(or \cite[Lemma 6.3]{Benson}) that if $xA$ is a prime ideal, then $A$ is a UFD.

Suppose $f,g\in A$ with $fg\in xA$. Since $A$ is graded we may assume $f$ and $g$ are homogeneous.
Since $x\field[V]$ is prime, without loss of generality,  we may assume $f\in x\field[V]$. Hence the lead monomial
$\lm(f)$ is divisible by $x$. $\mathcal{B}$ is a SAGBI basis for $A$ and $f\in A$. Thus $f$ subducts to $0$. Using the
grevlex order with $x$ small, every monomial of degree $\deg(f)$, less than $\lm(f)$, is divisible by $x$.
Thus at each stage of the
subduction, there is a factor of $x$. Hence $f\in xA$ and $xA$ is a prime ideal.
\end{proof}

The usual SAGBI bases algorithms applied to $\mathcal{B}$ proceeds by subducting \tat s and adjoining
non-zero subductions to produce a sequence
$\mathcal{B}=\mathcal{B}_0\subseteq\mathcal{B}_1\subseteq\mathcal{B}_2\subseteq\cdots$ with each $\mathcal{B}_i$ a
generating set for $A$ (see, for example \cite[Chapter 11]{Sturmfels}). The {\it SAGBI/Divide-by-$x$}
algorithm is an extension of the SAGBI algorithm: if a non-zero subduction $f$ has lead monomial
$x^my^E$, then $fx^{-m}$ is adjoined rather than $f$. The SAGBI/Divide-by-$x$ algorithm produces
a sequence of sets $\mathcal{B}=\mathcal{B}_0\subseteq\mathcal{B}_1\subseteq\mathcal{B}_2\subseteq\cdots$ and
a sequence of ring extensions $A\subseteq A_1 \subseteq A_2 \subseteq \cdots$, with $\mathcal{B}_i$
a generating set of $A_i$.

\begin{theorem} \label{sdx} Suppose that $x\in\mathcal{B}$, $\mathcal{B}$ generates $A$,
and $A[x^{-1}]=\field[V]^G[x^{-1}]$.
Further suppose there exists $h_1,\ldots,h_{n-1}\in\mathcal{B}$ such that $\lm(h_i)=y_i^{a_i}$ for positive integers $a_i$.
Then the SAGBI/Divide-by-$x$ algorithm applied to $\mathcal{B}$ terminates with a SAGBI basis for $\field[V]^G$.
\end{theorem}
\begin{proof}
Let $H$ denote the lead term algebra of $\field[x,h_1,\ldots,h_{n-1}]$, i.e.,
$H=\field[x,y_1^{a_1},\ldots,y_{n-1}^{a_{n-1}}]$. Then $\field[V]$ is a finite module over the Noetherian ring $H$.
Let $B_i$ denote the algebra generated by the lead terms of the elements of $\mathcal{B}_i$. Thus
$B_0\subseteq B_1\subseteq B_2 \subseteq \cdots$ is an ascending chain of $H$-modules in the
Noetherian $H$-module $\field[V]$. Hence this sequence terminates, say with $B_j$. Thus we have a SAGBI basis
$\mathcal{B}_j$ for the algebra $A_j$. Since $\{x,h_1,\ldots,h_{n-1}\}\subset A\subseteq A_j$ is a homogeneous
system of parameters for $\field[V]^G$ and  $A_j[x^{-1}]=A[x^{-1}]=\field[V]^G[x^{-1}]$,
the hypotheses of
Theorem~\ref{genthm} are met and $\mathcal{B}_j$ is a SAGBI basis for $\field[V]^G$.
\end{proof}

\section{Two dimensional representations}\label{2drep_sec}

Let $\field$ denote a field of characteristic $p$ and consider a finite subgroup $W$ of the additive group
$(\field,+)$. The order of $W$ is $p^r$ for some non-negative integer $r$ and $(W,+)$ is isomorphic to
the elementary abelian $p$-group $((\bz/p)^r,+)$.
Choosing an isomorphism from $(\bz/p)^r$ to $W$ is equivalent to choosing a basis for $W$ as a
vector space over the finite field $\fp$. Define a group homomorphism $\rho:\field \to GL_2(\field)$ by
$$\rho(c):=\left[ \begin{array}{cc} 1&c\\ 0&1 \end{array}\right].$$
The restriction of $\rho$ to $W$ gives a representation of $W$.

\begin{Prop} Two subgroups of $GL_2(\field)$ of the form $\rho(W)$ and $\rho(W')$ are conjugate
if and only if $W'=\alpha W$ for some $\alpha\in \field^*$.
\end{Prop}
\begin{proof}
Since $\rho(W)$ and $\rho(W')$ have the same socle, only an upper-triangular matrix can conjugate
$\rho(W)$ to $\rho(W')$. For an invertible matrix
$$M=
\left[ \begin{array}{cc} m_{11}&m_{12}\\0&m_{22}  \end{array}\right],$$
define $\alpha:=\frac{m_{11}}{m_{22}}\in\field^*$.
Then $M\rho(c)M^{-1}=\rho(\alpha c)$ and $W'=\alpha W$.
\end{proof}

Every faithful two dimensional $\field$-representation of $(\bz/p)^r$ is equivalent
to $\rho\circ\varphi$ for some group monomorphism $\varphi:(\bz/p)^r\to\field$ and therefore has image
$\rho(W)$ for some finite subgroup $W$. Two such representations, say $\rho\circ \varphi$ and
$\rho\circ\varphi'$,
are equivalent if and only if $\varphi'=\alpha\varphi$ for some $\alpha\in \field^*$.

Let $V_2$ denote the natural left $GL_2(\field)$-module, i.e., the space of two dimensional column vectors.
Identify the right $GL_2(\field)$-module $V_2^*$ with the space of two dimensional row vectors.
Taking $x=[0\; 1]$ and $y=[1\; 0]$, we have $x\rho(c)=x$ and $y\rho(c)=y+cx$. This action extends
to a right action of $(\field,+)$ on the polynomial algebra $\field[x,y]=\field[V_2]$ by degree preserving
algebra automorphisms. For any finite subgroup $(W,+)$ of $(\field,+)$ we denote the ring of invariants
of the restriction of $\rho$ to $W$ by $\field[V_2]^W$. Clearly $x\in\field[V_2]^W$.
A second invariant is given by the orbit product
$$N_W(y):=\prod_{c\in W}y\rho(c)=\prod_{c\in W}(y+cx).$$
Using the graded reverse lexicographic order with $x<y$, the leading term of $N_W(y)$ is $y^{p^r}$.
Thus $\{x, N_W(y)\}$ is a homogeneous system of parameters with $\deg(x)\deg(N_W(y))=|W|$.
Hence  applying \cite[3.7.5]{DK} gives the following proposition.

\begin{Prop} $\field[V_2]^W=\field[x,N_W(y)]$.
\end{Prop}

Define
$$F_W(t):=\prod_{c\in W}(t-c) \in \field[t].$$
Thus $F_W$ is the monic polynomial whose roots are the elements of $W$.
Note that for $\alpha\in \field^*$ we have $F_{\alpha W}(t)=\alpha^{|W|}F_W(t/\alpha)$.
Furthermore, $N_W(y)= x^{|W|}F_W(y/x)$ and $N_{\alpha W}(y)=(\alpha x)^{|W|}F_W(y/(\alpha x))$.
Therefore the same monomials appear with non-zero coefficients in both $N_W(y)$ and $N_{\alpha W}(y)$.

Choose an $\field_p$-vector space basis $\{c_1,\ldots,c_r\}$ for $W\subseteq \field$.
Since $\field$ is an $\fp$-algebra, taking $\psi(x_i)=c_i$
defines an $\fp$-algebra map $\psi:\fp[x_1,\ldots,x_r]\to \field$.
A change of basis for $W$ corresponds to the action of an element of $GL_r(\fp)$ on
$\fp[x_1,\ldots,x_r]$. Thus the restriction of $\psi$ to $\fp[x_1,\ldots,x_r]^{GL_r(\fp)}$ is determined
by $W$ and is independent of the choice of basis of $W$. We denote the restriction by $\psi_W$.

Let $U$ denote the $\fp$-span of $\{x_1,\ldots,x_r\}$. Recall that the
Dickson invariants, $d_1,\ldots, d_r$, may be defined as the coefficients of the polynomial
$$D(t):=\prod_{v\in U}(t-v)=t^{p^r}+\sum_{i=0}^{r-1} d_{r-i}t^{p^i}\in \fp[x_1,\ldots,x_r][t]$$
and that $\fp[x_1,\ldots,x_r]^{GL_r(\fp)}=\fp[d_1,\ldots,d_r]$(see \cite{Dickson} or \cite{Wilkerson}).
Define $d_i(W):=\psi_W(d_i)\in\field$.

\begin{theorem} $N_W(y)=y^{p^r}+\sum_{i=0}^{r-1}d_{r-i}(W)y^{p^i}x^{p^r-p^i}$.
\end{theorem}

\begin{proof}
Applying $\psi$ to the polynomial $D(t)$ gives $F_W(t)$. Hence
\begin{eqnarray*}
N_W(y)&=&x^{p^r}\psi(D(y/x))=x^{p^r}\left((y/x)^{p^r}+\sum_{i=0}^{r-1} d_{r-i}(W)(y/x)^{p^i}\right)\\
     &=& y^{p^r}+\sum_{i=0}^{r-1}d_{r-i}(W)y^{p^i}x^{p^r-p^i}
\end{eqnarray*}
as required.
\end{proof}

\begin{remark}
For $\field=\fp(x_1,\ldots,x_r)$ and $W={\rm Span}_{\fp}\{x_1,\ldots,x_r\}$, the map $\psi$ is injective
and $d_i(W)\not=0$ for $i=1,\ldots,r$. However, if $W=\field_{p^r}\subseteq \field$, then
$F_W(t)=t^{p^r}-t$, giving $d_r(W)=-1$ and $d_i(W)=0$ for $1\leq i<r$.
\end{remark}

A two dimensional  representation of $(\zp)^r$ is given by a point $(c_1,\ldots,c_r)\in\field^r$:
$(\zp)^r=\langle e_1,\ldots,e_r\rangle$ and $e_i\mapsto c_i$.
The representation is faithful if $\dim_{\fp}({\rm Span}_{\fp}\{c_1,\ldots,c_r\})=r$.
Two points $v,v'\in\field^r$ give equivalent representations
if and only if $v'=\alpha v$ for some $\alpha\in\field^*$. Thus equivalence classes of
non-trivial representations are parameterised
by points in projective space $P(\field^r)$. The action of the automorphism group ${\rm Aut}((\zp)^r)=GL_r(\fp)$
corresponds to changing generators for $(\zp)^r$ and permutes representations while preserving the ring of invariants. Thus
the rings of invariants for non-trivial representations are parameterised by the projective variety associated
to $\field^r\rightquot GL_r(\fp)$.  The coordinate ring of this variety is $\field[x_1,\ldots,x_r]^{GL_r(\fp)}=\field[d_1,\ldots,d_r]$ and thus it is just a weighted projective space of dimension $r-1$ (with weights $p^r-p^i$ for $i=0,1,\dots,r-1$).
The representation given by $v\in\field^r$ is faithful if and only if $d_r(v)\not=0$.
To see this, observe that $d_r$ is the product of all non-trivial $\fp$-linear combinations of
$\{x_1,\ldots,x_r\}$.

\section{Symmetric square representations}\label{ss_sec}

In this section we assume $p>2$.
Let $V_3$ denote the representation of $(\field, +)$ dual to the symmetric square
of $V_2^*$. Since $x^2\rho(c)=x^2$, $(xy)\rho(c)=xy + cx^2$ and $y^2\rho(c)=(y+cx)^2$,
$V_3$ is the left $(\field,+)$-module given by
$$\rho_2(c):=\left[ \begin{array}{ccc} 1&2c& c^2\\0&1&c \\0&0&1 \end{array}\right].$$
Taking $x=[0\;0\; 1]$, $y=[0\; 1\; 0]$, and $z=[1\;0\;0]$, we have $x\rho_2(c)=x$, $y\rho_2(c)=y+cx$, and
$z\rho_2(c)=z+2cy+c^2x$. Note that ``multiplication by $x$'' embeds $V_2^*$ as a submodule of $V_3^*$,
justifying our convention of using $x$ and $y$ to denote elements of both $V_2^*$ and $V_3^*$.
Dual to this embedding, we have the $(\field,+)$-module surjection
$V_3\to V_3/V_3^{\field}\cong V_2$. The action of $(\field,+)$ on $V_3^*$ extends
to a right action on the polynomial algebra $\field[x,y,z]=\field[V_3]$ by degree preserving
algebra automorphisms. For any finite subgroup $(W,+)$ of $(\field,+)$ we denote the ring of invariants
of the restriction of $\rho_2$ to $W$ by $\field[V_3]^W$. It is clear that $x\in \field[V_3]^W$.
A simple calculation shows that $$\delta:=y^2-xz\in\field[V_3]^W.$$ Since
$$\prod_{c\in W}y\rho_2(c)=\prod_{c\in W}(y+cx)=N_W(y),$$
we also have $N_W(y)\in \field[V_3]^W$. Define
$$N_W(z):=\prod_{c\in W}z\rho_2(c)=\prod_{c\in W}(z+2cy+c^2x).$$
The rest of this section is devoted to the proof that $\field[V_3]^W$ is generated
by $\mathcal{G}:=\{x, \delta, N_W(y), N_W(z)\}$.
We will use the graded reverse lexicographic order with $z>y>x$ and show that $\mathcal{G}$ is a SAGBI basis for
$\field[V_3]^W$ with respect to this order.
With this order we have
$\lt(\delta)=y^2$, $\lt(N_W(y)) = y^{p^r}$ and $\lt(N_W(z))=z^{p^r}$.  Thus there is a single
non-trivial \tat\ among the elements of $\mathcal{G}$:
$$\delta^{p^r}-N_W(y)^2= \left(y^{2p^r}-(xz)^{p^r}\right)-\left( y^{p^r}+\sum_{i=0}^{r-1}d_{r-i}(W)y^{p^i}x^{p^r-p^i}\right)^2.$$
Note that
\begin{eqnarray*}
N_W(y)^2-y^{2p^r}
&=&2\sum_{0\leq i<j\leq r}d_{r-i}(W)d_{r-j}(W)y^{p^i+p^j}x^{2p^r-p^i-p^j}\\
&+&\sum_{i=0}^{r-1}d_{r-i}(W)^2y^{2p^i}x^{2(p^r-p^i)}
\end{eqnarray*}
is a polynomial in $x$ and $y^2$. Define $H_W(s,t)\in \field[s,t]$ so that
$H_W(x,y^2)=N_W(y)^2-y^{2p^r}$.

\begin{lemma} $\delta^{p^r}-N_W(y)^2+x^{p^r}N_W(z)+H_W(x,\delta)=0$.
\label{rel_lem}
\end{lemma}
\begin{proof}
Consider $F:=\delta^{p^r}-N_W(y)^2+H_W(x,\delta)$.
Note that the coefficient of $z^{p^r}$ in $F$ is $-x^{p^r}$.
Also, working modulo the ideal
in $\field[V_3]$ generated by $z$ we have $F\equiv_{(z)} y^{2p^r}-N_W(y)^2+H_W(x,y^2)=0$.
Thus $z$ divides $F$. However, since $F\in\field[V_3]^W$, each element in the $W$-orbit
of $z$ divides $F$. By definition, $N_W(z)$ is the product of the elements in the $W$-orbit of $z$.
Also, observe that if $\alpha$ and $\beta$ are elements in the $W$-orbit of $z$, then $\alpha$ divides $\beta$
if an only if $\alpha=\beta$. Therefore $N_W(z)$ divides $F$.
Since  $-x^{p^r}$ is the coefficient of $z^{p^r}$ in $F$, we must have $F=-x^{p^r}N_W(z)$, as required.
\end{proof}

Let $A$ denote the subalgebra of $\field[V_3]^W$ generated by $\mathcal{G}$. Using the relation given by the lemma,
the sole non-trivial \tat\ subducts to $0$, giving the following.

\begin{theorem} $\mathcal{G}$ is SAGBI basis for $A$.
\end{theorem}

\begin{theorem}\label{spthm} $\field[V_3]^W$ is the hypersurface generated by $\{x, \delta, N_W(y), N_W(z)\}$
subject to the relation $\delta^{p^r}-N_W(y)^2+x^{p^r}N_W(z)+H_W(x,\delta)=0$. Furthermore,
this generating set is a SAGBI with respect to the graded reverse lexicographic order with $z>y>x$.
\end{theorem}
\begin{proof}
Note that $\lt(\delta)=y^2$ and $\lt(N_W(z))=z^{p^r}$. Thus $(x,\delta, N_W(z))\field[V_3]$ is a zero-dimensional ideal
and $\{x,\delta,N_W(z)\}$ is a homogeneous system of parameters.
Hence $A\subseteq\field[V_3]^W$ is an integral extension. Since $\field(x,y)^W=\field(x,N_W(y))$ and $\delta$ is degree 1 in $z$,
applying \cite[Theorem 2.4]{CampChuai}  gives $\field(V_3)^W=\field(x,N_W(y),\delta)$ (see also \cite{Kang}).
Furthermore, since the  coefficient of $z$ in $\delta$ is $-x$, we have
$\field[V_3]^W[x^{-1}]=\field[x,N_W(y),\delta][x^{-1}]=A[x^{-1}]$.
(Note that using the relation given in Lemma~ \ref{rel_lem}, we
obtain an explicit expression for $N_W(z) \in \field[x, \delta, N_W(y)][x^{-1}].)$
Applying Theorem~\ref{genthm} gives $A=\field[V_3]^W$.
\end{proof}

\begin{remark} A proof of Theorem~\ref{spthm} for the special case $W=\field_{p^r}$ is given in Section~2 of \cite{HS2}.
\end{remark}

\section{Classifying the three dimensional representations}\label{cl3d_sec}
In this section we describe all three dimensional representations of $G:=(\zp)^r$.
We first sort the non-trivial representations $V$ of $G$ by the dimensions of the factors in the socle series.

\noindent
Type $(2,1)$: $\dim_{\field}(V^G)=2$ and $\dim_{\field}((V/V^G)^G)=1$. The image of the
representation is of the form
$$\left\{
\left[ \begin{array}{ccc} 1&0&c_1 \\0&1&c_2 \\0&0&1 \end{array}\right]\;\middle\vert\; (c_1,c_2)\in U\right\}$$
for some finite subgroup $U\leq(\field^2,+)$. By \cite[Theorem~3.9.2]{CW} (or \cite{ls}),
the rings of invariants for these representations
are polynomial algebras. Note that this case includes Stong's example (see \cite[\S 8.1]{CW})
so the image is not necessarily a Nakajima group. (A Nakajima group is a subgroup of the
unipotent upper triangular matrices whose ring of invariants is generated by the orbit products
of linear forms, see \cite[Chapter~8]{CW} for details.)

\noindent
Type $(1,2)$: $\dim_{\field}(V^G)=1$ and $\dim_{\field}((V/V^G)^G)=2$.
The image of the representation is of the form
$$\left\{
\left[ \begin{array}{ccc} 1&c_1&c_2 \\0&1&0 \\0&0&1 \end{array}\right]\;\middle\vert\; (c_1,c_2)\in U\right\}$$
for some finite subgroup $U\leq(\field^2,+)$. Here $\field[V]^G=\field[x,y,N_U(z)]$
where
$$N_U(z):=\prod_{(c_1,c_2)\in U}(z+c_1y+c_2x)$$
by  \cite[3.7.5]{DK}.

\noindent
Type $(1,1,1)$: $\dim_{\field}(V^G)=1$ and $\dim_{\field}((V/V^G)^G)=1$. In this case the image of the representation
contains at least one element whose Jordan form consists of a single Jordan block; hence $p>2$.
This case includes the symmetric square
representations of Section~\ref{ss_sec}. Define a group homomorphism $\sigma:\field^2\to GL_3(\field)$ by
$$\sigma(c_1,c_2):=\left[ \begin{array}{ccc} 1&2c_1&c_1^2+c_2 \\0&1&c_1 \\0&0&1 \end{array}\right].$$

\begin{Prop}
For any representation of type $(1,1,1)$, there exists a choice of basis for which the image of the representations
 is given by $\sigma(U)$
for some finite subgroup $U\leq(\field^2,+)$.
\end{Prop}
\begin{proof}
Since the representation is of type $(1,1,1)$ there is at least one element whose Jordan form consists of a single Jordan block. Note that this element determines the socle series
of the representation. Choose a basis for $V$ so that the matrix representing this element
is $\sigma(1,0)$. With respect to this basis, the other group elements are represented by upper triangular unipotent matrices. Furthermore, these matrices must commute with
$\sigma(1,0)$. A straight forward computation shows that any upper triangular unipotent matrix which commutes with $\sigma(1,0)$ is in the image of $\sigma$.
\end{proof}

\begin{remark}
If $V$ is a decomposable representation of type $(2,1)$ then there exists a choice of basis for which the image of the representation is contained in $\{\sigma(0,c)\mid c\in\field\}$, 
the centre of the upper triangular unipotent group.
\end{remark}

\begin{Prop}
Two subgroups of $GL_3(\field)$ of the form $\sigma(U)$ and $\sigma(U')$ are conjugate if and only if
there exist
$\alpha\in \field^*$ and $\gamma\in\field$ such that
$$U'=\left\{\alpha(c_1,\gamma c_1+\alpha c_2)\mid (c_1,c_2)\in U\right\}.$$
\end{Prop}
\begin{proof} Since the subgroups $\sigma(U)$ and $\sigma(U')$ have the same socle series,
only an upper triangular matrix can conjugate $\sigma(U)$ to $\sigma(U')$. For an invertible
matrix
$$M:=\left[\begin{array}{ccc} m_{11}&m_{12}&m_{13}\\0&m_{22}&m_{23}\\0&0&m_{33}  \end{array}\right],$$
define $\alpha:=m_{11}/m_{22}\in\field^*$ and $\gamma:=(m_{12}/m_{22})-2(m_{23}/m_{33})\in\field$.
The condition $M\sigma(c_1,c_2)M^{-1}\in\sigma(\field)$ forces $m_{22}/m_{33}=\alpha$ and
gives $M\sigma(c_1,c_2)M^{-1}=\sigma(\alpha c_1,\alpha^2c_2+\alpha\gamma c_1)$.
\end{proof}
We encode this conjugation as a left action of $\field\ltimes\field^*$ on $\field^2$:
$$(\gamma,\alpha)\cdot\left[\begin{array}{c} c_1\\ c_2 \end{array}\right]=
\left(\begin{array}{cc} \alpha &0\\ \alpha\gamma &\alpha^2
\end{array}\right)\left[\begin{array}{c} c_1\\ c_2 \end{array}\right],$$
where $\field^*$ acts on $\field$ by
multiplication.
A matrix $$M=\left(
\begin{array}{cccc}
c_{11}& c_{12} & \cdots & c_{1r} \\
c_{21}& c_{22} & \cdots & c_{2r}
\end{array}\right) \in \field^{2\times r}
$$
determines a three dimensional representation $V_M$ of $G=(\zp)^r=\langle e_1,\ldots,e_r\rangle$ by
$e_i\mapsto \sigma(c_{1i},c_{2i})$.
The right action of $GL_r(\fp)$ on $\field^{2\times r}$ corresponds to
changing generators for $(\zp)^r$ and permutes representations while preserving the
ring of invariants. The left action of $\field\ltimes\field^*$ on $\field^{2\times r}$
corresponds to a change of basis for $V_M$. The rings of invariants for representations of this
form are parameterised by $(\field\ltimes\field^*)$-orbits in the variety
$\field^{2\times r}\rightquot GL_r(\fp)$.
The representation $V_M$ is of type $(1,1,1)$ if $c_{1i}\not=0$ for some $i$.
If $V_M$ is of type $(1,1,1)$ and $c_{2i}=0$ for all $i$, then $V_M$ is a symmetric square representation.

We choose a basis $\{x,y,z\}$ for $V_M^*$ so that $x\sigma(c_1,c_2)=x$, $y\sigma(c_1,c_2)=y+c_1x$
and $z\sigma(c_1,c_2)=z+2c_1y+(c_1^2+c_2)x$. Let $U$ denote the $\fp$-span of the columns of $M$.
Then $\field[V_M]^G=\field[x,y,z]^{\sigma(U)}$. For any $f\in\field[x,y,z]$, let $N_G(f)$ denote
the product of the elements of $G$-orbit of $f$.

\begin{theorem}\label{ffthm}
If $\deg(N_G(y))\deg(N_G(\delta))=2|G|$, then
$$\field[V_M]^G[x^{-1}]=\field[x,N_G(y),N_G(\delta)][x^{-1}].$$
\end{theorem}
\begin{proof}
Since $\delta=y^2-xz$, it is clear that $\field[x,y,z][x^{-1}]=\field[x,y,\delta/x][x^{-1}]$.
Observe that $\delta\sigma(c_1,c_2)=\delta-c_2x^2$. Thus $G$ acts on $\field[x,y,\delta/x]$
by degree preserving algebra automorphisms. By hypothesis,  $\deg(N_G(y))\deg(N_G(\delta/x))=|G|$.
Therefore, by \cite[3.7.5]{DK},
$\field[x,y,\delta/x]^G=\field[x,N_G(y), N_G(\delta/x)]$.
Thus $\field[V_M]^G[x^{-1}]
=\field[x,N_G(y), N_G(\delta/x)][x^{-1}]
=\field[x,N_G(y), N_G(\delta)][x^{-1}]$.
\end{proof}


The  matrix $M\in\field^{2\times r}$ can be used to define an evaluation map $$\psi_M:\field[x_{ij}]\to \field$$
by $\psi_M(x_{ij})=c_{ij}$. We let $\psi_U$ denote the restriction of $\psi_M$ to $\field[\field^{2\times r}]^{GL_r(\fp)}$.
We expect $\psi_U$ to play a key r\^ole in the construction of generators for $\field[V_M]^G$.

\section{The field of fractions for the generic case}\label{field_section}

Consider the ring $\fp[x_{ij}]:=\fp[x_{1j},x_{2j}\mid j=1,2,\ldots r]$
and its quotient field $\rfield:=\fp(x_{1j},x_{2j}\mid j=1,2,\ldots r)$.
In this section, we work over $\rfield$.  
The action of $G=(\zp)^r=\langle e_1,\ldots,e_r\rangle$ is  
given by $e_j\mapsto\sigma(x_{1j},x_{2j})$.
Therefore, we assume $p>2$.  
Define a $2(r+1)\times r$ matrix by
$$\Gamma:=\left[\begin{array}{cccc}
x_{11} & x_{12}& \cdots &x_{1r} \\
x_{21} & x_{22}& \cdots &x_{2r} \\
x_{11}^p & x_{12}^p& \cdots &x_{1r}^p \\
x_{21}^p & x_{22}^p& \cdots &x_{2r}^p  \\
        &        &\vdots  &         \\
x_{11}^{p^r} & x_{12}^{p^r}& \cdots &x_{1r}^{p^r} \\
x_{21}^{p^r} & x_{22}^{p^r}& \cdots &x_{2r}^{p^r}
 \end{array}
\right].$$

For a subsequence $I=(i_1,\ldots,i_r)$ of $(1,2,\ldots,2r+2)$, let $\gamma_I\in \rfield$ denote
the associated $r\times r$ minor of $\Gamma$. Note that $\gamma_I$ is invariant under the natural right action
of $SL_r(\fp)$. Form a $2(r+1)\times (r+1)$ matrix $\widetilde{\Gamma}$ by augmenting $\Gamma$ with the column
$${\bf v}=[y/x, -\delta/x^2, (y/x)^p, (-\delta/x^2)^p, \ldots, (y/x)^{p^r},  (-\delta/x^2)^{p^r}]^T.$$
For a subsequence $J=(j_1,\ldots,j_{r+1})$ of $(1,2,\ldots,2r+2)$, let $\tilde{f}_J\in\rfield[x,y,z][x^{-1}]$
denote the associated $(r+1)\times (r+1)$ minor of $\widetilde{\Gamma}$. Let $f_J$ denote the element of
$\rfield[x,y,z]$ constructed by minimally clearing the denominator of $\tilde{f}_J$.
Observe that 
$y\Delta_j=xx_{1j}$ and $\delta\Delta_j =-x^2x_{2j}$, where $\Delta_j:=e_j-1\in\rfield G$. Therefore
$${\bf v}\Delta_j=[\Gamma_{1j},\Gamma_{2j},\ldots, \Gamma_{(2r+1)j}]^T,$$
the $j^\text{th}$ column of $\Gamma$.
Thus $\tilde{f}_J\in \rfield[x,y,z]^G[x^{-1}]$ and $f_J\in\rfield[x,y,z]^G$.

\begin{lemma}\label{biggen} $\rfield[x,y,z]^G[x^{-1}]
=\rfield[x,f_{(1,3,5,\ldots, 2r+1)},f_{(1,2,3,5,\ldots,2r-1)}][x^{-1}]$.
\end{lemma}
\begin{proof}
Since $f_{(1,3,5,\ldots, 2r+1)}$ is a scalar multiple of $N_G(y)$ and
$f_{(1,2,3,5,\ldots,2r-1)}$ is an invariant which is degree 1 in $z$,
applying \cite[Theorem 2.4]{CampChuai} shows that $\rfield(x,y,z)^G$ is generated by
$\{x,f_{(1,3,\ldots, 2r+1)},f_{(1,2,3,5,\ldots,2r-1)}\}$.
Furthermore, the coefficient of $z$ in $f_{(1,2,3,5,\ldots,2r-1)}$ is a scalar times a power of $x$.
Hence, we have a generating set once $x$ has been inverted.
\end{proof}

Observe that $\lm(f_{(1,3,\ldots, 2r+1)})=y^{p^r}$ and $\lm(f_{(1,2,3,5,\ldots,2r-1)})=y^{p^{r-1}}$.
Define $f_1:=f_{(1,2,3,\ldots, r+1)}$ and $s:=\lceil r/2 \rceil$. For $r$ odd, define
$f_2:=f_{(1,2,\ldots,r,r+2)}$ and observe that $\lm(f_1)=y^{2p^{s-1}}$ and $\lm(f_2)=y^{p^s}$.
For $r$ even, define
$$f_2:=\frac{\gamma_{(1,2,\ldots,r)}f_{(1,2,\ldots,r,r+2)}+f_1^2}{2x^{p^s-2p^{s-1}}}.$$
In this case $\lm(f_1)=y^{p^s}$ and a straight-forward calculation gives
$$\lt(f_2)=\gamma_{(1,2,\ldots,r)}\gamma_{(1,2,\ldots,r-1,r+1)}y^{p^s+2p^{s-1}}.$$
The rest of this section is devoted to the proof of the following.

\begin{theorem}\label{fieldthm}  $\rfield[x,y,z]^G[x^{-1}]=\rfield[x,f_1,f_2][x^{-1}]$.
\end{theorem}

For a subsequence $K=(k_1,k_2,\ldots,k_{r+2})$ of $(1,2,\ldots,2r+2)$, let $K_{\ell}$ denote the
subsequence of $K$ formed by omitting $\ell$ and let $K_{\ell,m}$ denote the subsequence formed by
omitting $\ell$ and $m$.

\begin{lemma}\label{plucker} For any subsequence $(\ell_1,\ell_2,\ell_3)$ of $K$,
$$(-1)^{\epsilon_1}\gamma_{K_{\ell_1,\ell_2}}\tilde{f}_{K_{\ell_3}}+
(-1)^{\epsilon_2}\gamma_{K_{\ell_2,\ell_3}}\tilde{f}_{K_{\ell_1}}+
(-1)^{\epsilon_3}\gamma_{K_{\ell_1,\ell_3}}\tilde{f}_{K_{\ell_2}}=0$$
for some choice of $\epsilon_i\in\{0,1\}$.
\end{lemma}
\begin{proof} Form a matrix $\Lambda$ by adding the row $[0,0,0,\ldots,0,1]$ to the
bottom of $\widetilde{\Gamma}$. Thus the minors of $\Lambda$ which include the last row
are, up to sign, minors of $\Gamma,$ and the minors of $\Lambda$ which do not include the last row
are minors of $\widetilde{\Gamma}$.
Consider the Pl\"ucker relation for $\Lambda$ determined
by the sequence $K$, of length $r+2$, and the sequence of length $r$ formed by omitting
$\ell_1,\ell_2,\ell_3$ from $K$ and adding $2r+3$. All but three of the terms in this
Pl\"ucker relation are zero and the non-zero terms give the required relation.
\end{proof}

\begin{lemma}\label{powers} For $J$ a length $r+1$ subsequence of $(1,2,\ldots,2r+2)$,
$$\tilde{f}_J\in {\rm Span}_{\rfield}\{
\tilde{f}^{p^i}_{(1,2,3,\ldots, r+1)},\tilde{f}^{p^i}_{(2,3,\ldots,r+2)}\mid i=0,1,2,\ldots\}.$$
\end{lemma}
\begin{proof} Since the $p^{th}$ power map is $\fp$-linear,
$\tilde{f}_J^p=\tilde{f}_{(j_1+2,j_2+2,\ldots,j_{(r+1)}+2)}$.

The proof is by induction on $t=j_{r+1}-j_1$, starting with $t=r$.
If $t=r$ and $j_1$ is even, say $j_1=2i+2$, then $\tilde{f}_J=\tilde{f}^{p^i}_{(2,3,\ldots,r+2)}$.
If $t=r$ and $j_1$ is odd, say $j_1=2i+1$, then $\tilde{f}_J= \tilde{f}^{p^i}_{(1,2,3,\ldots, r+1)}$.

Suppose $t>r$. Choose an integer $m\not\in\{j_1,\ldots,j_{r+1}\}$ with $j_1<m<j_{r+1}$. Insert $m$ into
the sequence $J$ to produce a sequence $K$ of length $r+2$. Apply Lemma~\ref{plucker} to the
subsequence $(j_1,m,j_{r+1})$ of $K$. This allows us to write $\tilde{f}_J$ as an $\rfield$-linear combination
of $\tilde{f}_{K_{j_1}}$ and  $\tilde{f}_{K_{j_{(r+1)}}}$, both of which lie in ${\rm Span}_{\rfield}\{
\tilde{f}^{p^i}_{(1,2,3,\ldots, r+1)},\tilde{f}^{p^i}_{(2,3,\ldots,r+2)}\mid i=1,2,\ldots\}$ by induction.
\end{proof}

Using Lemma~\ref{powers} and Lemma~\ref{biggen}, we see that
 $$\rfield[x,y,z]^G[x^{-1}]
=\rfield[x,\tilde{f}_{(1,2,\ldots,r+1)},\tilde{f}_{(2,3,\ldots,r+2)}][x^{-1}].$$ Applying Lemma~\ref{plucker} and clearing denominators
completes the proof of Theorem~\ref{fieldthm}.

\begin{remark} It follows from Theorem~\ref{sdx} that applying the SAGBI/Divide-by-$x$ algorithm to
$\{x,f_1,f_2,N_G(z)\}$ will produce a SAGBI basis for $\rfield[V]^G$.
\end{remark}

\section{The invariants for rank two dimension three}\label{r2d3_sec}

In this section $G=(\zp)^2=\langle e_1,e_2\rangle$. We start with the generic representation over
$\rfield:=\fp(x_{11},x_{12},x_{21},x_{22})$ given by $e_i\mapsto \sigma(x_{1i},x_{2,i})$; hence, we assume $p>2$.
Consider $\mathcal{B}:=\{x,f_1,f_2,N_G(z)\}$ with $f_1$ and $f_2$ defined as in Section~\ref{field_section}.
Thus $\lt(f_1)=\gamma_{12}y^p$ and $\lt(f_2)=\gamma_{12}\gamma_{13}y^{p+2}$. There is a single
non-trivial \tat: $(f_1^{p+2},f_2^p)$.

For $p=3$, define
$$\widetilde{N}:=\gamma_{13}^3f_1^5-\gamma_{12}^2f_2^3+c_1xf_1^3f_2 +c_2x^3f_1^4 +c_3x^4f_2f_1^2+c_4x^5f_2^2$$
with
$c_1:=\gamma_{13}^3$,
$c_2= \gamma_{12}\gamma_{23}^3+(\gamma_{13}^3/\gamma_{12})^2$,
$c_3=\gamma_{12}\gamma_{13}^2\gamma_{14}-\gamma_{12}\gamma_{23}^3-(\gamma_{13}^3/\gamma_{12})^2$ and
$c_4=\gamma_{12}^3\gamma_{34}-\gamma_{12}^2\gamma_{13}^2\gamma_{14}+\gamma_{12}\gamma_{23}^3+(\gamma_{13}^3/\gamma_{12})^2$.

For $p>3$, define
$$\widetilde{N}:=\gamma_{13}^pf_1^{p+2}-\gamma_{12}^2f_2^p+c_1x^{p-2}f_1^pf_2+c_2x^pf_1^{p+1}
+c_3x^{2p-2}f_2f_1^{p-1}+c_4x^{2p-1}f_2^{(p+1)/2}f_1^{(p-3)/2}$$
with
$c_4=\gamma_{12}^2\gamma_{13}^{(p-3)/2}\left(\gamma_{13}\gamma_{24}-\gamma_{23}\gamma_{14}\right)$,
$c_3=\gamma_{12}(\gamma_{14}\gamma_{13}^{p-1}-\gamma_{23}^p)$,
$c_2= \gamma_{12}\gamma_{23}^p$ and
$c_1=-2\gamma_{13}^p$.

\begin{lemma}\label{rank2-gen-lt}
$\lt(\widetilde{N})=-\frac{1}{2}\gamma_{12}^{2p+2}x^{2p}z^{p^2}$.
\end{lemma}
\begin{proof} For $p=3$, verifying the result is a Magma calculation. Suppose $p>3$.
We work modulo the ideal in $\rfield[x,y,z]$ generated by $x^{2p+1}$ and
$yx^{2p}$, which we denote by $\mathfrak{n}$. By definition,
$f_1=\gamma_{12}y^p+\gamma_{13}\delta x^{p-2}+\gamma_{23}yx^{p-1}$.
For $p\geq 5$, we have ${p^2-2p}>2p+1$ and $2p+p-4>2p$.
Thus $f_1^p\equiv_{\mathfrak{n}}\gamma_{12}^py^{p^2}$
and $x^{2p-2}f_1\equiv_{\mathfrak{n}}\gamma_{12}x^{2p-2}y^p$.
By definition $2x^{p-2}f_2=\gamma_{12}f_{124}+f_1^2=f_1^2-\gamma_{12}^2\delta^p
+\gamma_{12}\gamma_{14}\delta x^{2p-2}+\gamma_{12}\gamma_{24}yx^{2p-1}$.
Rearranging gives
\begin{eqnarray}\label{f2eqn}
f_1^2-\gamma_{12}^2\delta^p-2x^{p-2}f_2&=&-\gamma_{12}\left(\gamma_{14}\delta x^{2p-2}+\gamma_{24}y x^{2p-1}\right)
\end{eqnarray}
Solving for $f_2$ gives
\begin{eqnarray*}
f_2&=&\gamma_{12}\gamma_{13}\delta y^p+\gamma_{12}\gamma_{23}x y^{p+1}
+\frac{1}{2}\gamma_{12}^2x^2z^p+\frac{1}{2}\gamma_{13}^2\delta^2x^{p-2}+\gamma_{13}\gamma_{23}\delta y x^{p-1}\\
&+&\frac{1}{2}(\gamma_{12}\gamma_{14}\delta +\gamma_{23}^2 y^2)x^p+\frac{1}{2}\gamma_{12}\gamma_{24}y x^{p+1}.
\end{eqnarray*}
Thus $f_2^p\equiv_{\mathfrak{n}} \gamma_{12}^p\gamma_{13}^p\delta^py^{p^2}+\gamma_{12}^p\gamma_{23}^px^p y^{p^2+p}
+\frac{1}{2}\gamma_{12}^{2p}x^{2p}z^{p^2}$ and
$$x^{2p-1}f_2\equiv_{\mathfrak{n}}\gamma_{12}\gamma_{13}x^{2p-1}\delta y^p
\equiv_{\mathfrak{n}}\gamma_{12}\gamma_{13}x^{2p-1}y^{p+2}.$$
Substituting gives
\begin{eqnarray*}
\widetilde{N}&\equiv_{\mathfrak{n}}&
\gamma_{12}^py^{p^2}\left(\gamma_{13}^pf_1^2-\gamma_{12}^2\gamma_{13}^p\delta^p-2\gamma_{13}^px^{p-2}f_2
-\gamma_{12}^2\gamma_{23}^px^py^p+c_2x^pf_1\right)\\
&+&c_3x^{2p-2}f_2\left(\gamma_{12}y^p\right)^{p-1}
+c_4x^{2p-1}\left(\gamma_{12}\gamma_{13}y^{p+2}\right)^{(p+1)/2}\left(\gamma_{12}y^p\right)^{(p-3)/2}\\
&-&\frac{1}{2}\gamma_{12}^{2p+2}x^{2p}z^{p^2}.
\end{eqnarray*}
Using Equation~\ref{f2eqn} gives
\begin{eqnarray*}
\widetilde{N}&\equiv_{\mathfrak{n}}&
\gamma_{12}^py^{p^2}\left(-\gamma_{12}\gamma_{13}^p\left(\gamma_{14}\delta x^{2p-2}+\gamma_{24}yx^{2p-1}\right)
-\gamma_{12}^2\gamma_{23}^px^py^p+c_2x^pf_1\right)\\
&+&c_3x^{2p-2}f_2\left(\gamma_{12}y^p\right)^{p-1}
+c_4x^{2p-1}y^{p^2+1}\gamma_{12}^{p-1}\gamma_{13}^{(p+1)/2}
-\frac{1}{2}\gamma_{12}^{2p+2}x^{2p}z^{p^2}.
\end{eqnarray*}
Simplifying gives
\begin{eqnarray*}
\widetilde{N}&\equiv_{\mathfrak{n}}&
\gamma_{12}^{p+1}y^{p^2}\left(\delta x^{2p-2}\left(\gamma_{13}\gamma_{23}^p-\gamma_{14}\gamma_{13}^p\right)+
yx^{2p-1}\left(\gamma_{23}^{p+1}-\gamma_{24}\gamma_{13}^p\right)\right)\\
&+&c_3x^{2p-2}\gamma_{12}^py^{p^2}\left( \gamma_{13}\delta+\gamma_{23}xy\right)
+c_4x^{2p-1}y^{p^2+1}\gamma_{12}^{p-1}\gamma_{13}^{(p+1)/2}
-\frac{1}{2}\gamma_{12}^{2p+2}x^{2p}z^{p^2}\\
&\equiv_{\mathfrak{n}}&\gamma_{12}^py^{p^2}\delta x^{2p-2}
\left(\gamma_{12}\gamma_{13}\gamma_{23}^p-\gamma_{12}\gamma_{14}\gamma_{13}^p+c_3\gamma_{13}\right)\\
&+&\gamma_{12}^{p-1}x^{2p-1}y^{p^2+1}\left(\gamma_{12}^2\gamma_{23}^{p+1}-\gamma_{12}^2\gamma_{24}\gamma_{13}^p
+c_3\gamma_{12}\gamma_{23}+c_4\gamma_{13}^{(p+1)/2}\right)\\
&-&\frac{1}{2}\gamma_{12}^{2p+2}x^{2p}z^{p^2}.
\end{eqnarray*}
Substituting for $c_3$ and $c_4$ gives
$$\widetilde{N}\equiv_{\mathfrak{n}}
-\frac{1}{2}\gamma_{12}^{2p+2}x^{2p}z^{p^2},$$
as required.
\end{proof}

\begin{theorem}\label{rk2gen}For the generic rank two representation, $\mathcal{B}=\{x,f_1,f_2,N_G(z)\}$ is a SAGBI basis for $\rfield[V_3]^G$.
Furthermore, there is a single relation among the generators constructed by subducting $\gamma_{13}^pf_1^{p+2}-\gamma_{12}^2f_2^p$.
\end{theorem}
\begin{proof} Define $N:=\widetilde{N}/x^{2p}$.
It follows from Lemma~\ref{rank2-gen-lt} that $\mathcal{B}':=\{x,f_1,f_2,N\}$ is a SAGBI basis for the algebra it generates.
Applying Theorem~\ref{fieldthm} and Theorem~\ref{genthm} proves that $\mathcal{B}'$ is a SAGBI basis for $\rfield[V_3]^G$.
Thus the lead term algebra of $\rfield[V_3]^G$ is generated by $\{x,y^p,y^{p+2}, z^{p^2}\}$.
Since $\lm(\mathcal{B})=\{x,y^p,y^{p+2}, z^{p^2}\}$, $\mathcal{B}$ is also
a SAGBI basis for $\rfield[V_3]^G$. The single non-trivial \tat\ subducts to produce the relation.
\end{proof}

For $\field$ an arbitrary field of characteristic $p$, consider the representation $V_M$ of $(\zp)^2$ determined by
$$M:=\left[\begin{array}{cc}
c_{11}&c_{12}\\
c_{21}&c_{22}
\end{array}\right]$$
with $c_{ij}\in\field$, i.e., $e_i\mapsto \sigma(c_{1i},c_{2i})$ for $i=1,2$.
Let $\psi_M:\fp[x_{ij}]\to\field$ denote the evaluation map: $\psi_M(x_{ij})=c_{ij}$.
For $f\in\rfield[x,y,z]$, let $\of$ denote the element of $\field[x,y,z]$
constructed by minimally clearing denominators to obtain an element of $\fp[x_{ij}]$ and then and applying $\psi_M$ to the coefficients
of that element.
Note that we can also interpret elements of $\fp[x_{ij}]$ as regular functions on $\field^{2\times 2}$,
the space of representations of the given type.

\begin{theorem} If $\gamma_{12}\not=0$ and $\gamma_{13}\not=0$ then
$\{x,\overline{f_1},\overline{f_2}, N_G(z)\}$ is a SAGBI basis for $\field[V_M]^G$.
Furthermore, there is a single relation among the generators constructed by
subducting $\og_{13}^p\overline{f_1}^{p+2}-\og_{12}^2\overline{f_2}^p$.
\end{theorem}
\begin{proof}
We show that, in essence,  the proof of Theorem~\ref{rk2gen} survives evaluation.
Let $A$ denote the algebra generated by $\mathcal{B}':=\{x,\overline{f_1}, \overline{f_2},\overline{N}\}$.
Using Lemma~\ref{rank2-gen-lt}, $\mathcal{B}'$ is a SAGBI basis for $A$ and
the lead term algebra of $A$ is generated by $\{x,y^p,y^{p+2},z^{p^2}\}$.
Thus it is sufficient to show that $A[x^{-1}]=\field[V_M]^G[x^{-1}]$.
Note that $\overline{f_1}$ is degree 1 in $z$ with coefficient $-\og_{13}x^{p-1}$.
Hence $\field[x,\overline{f_1},N_G(y)][x^{-1}]=\field[V_M]^G[x^{-1}]$.
Therefore, to complete the proof, we need only show that $N_G(y)\in A[x^{-1}]$.

Using the notation of section~\ref{field_section}, $\of_{135}=\og_{13}N_G(y)$.
Recall that $f_1=f_{123}$ and that $\gamma_{12}f_{124}=2x^{p-2}f_2-f_1^2$.
Applying Lemma~\ref{plucker} to the subsequence $(1,4,5)$ of $(1,3,4,5)$ shows that
$\gamma_{34}\widetilde{f}_{135}$ is an $\fp$-linear combination of
$\gamma_{35}\widetilde{f}_{134}$  and $\gamma_{13}\widetilde{f}_{345}$.
However, $\gamma_{34}=\gamma_{12}^p$ , $\gamma_{35}=\gamma_{13}^p$,
and $\widetilde{f}_{345}=\widetilde{f}_{123}^p$.
Applying Lemma~\ref{plucker} to the subsequence $(2,3,4)$ of $(1,2,3,4)$ shows that
$\gamma_{12}\widetilde{f}_{134}$ is an $\fp$-linear combination of
$\gamma_{13}\widetilde{f}_{124}$  and $\gamma_{14}\widetilde{f}_{123}$.
Thus $\gamma_{12}^pf_{135}$ can be written as a polynomial in $x$, $f_1$ and $f_2$
with coefficients in $\fp[x_{ij}]$. Thus $N_G(y)\in A$, as required.
\end{proof}

For a representation with $\gamma_{13}=0$ and $\gamma_{12}\not=0$, define
$h:=\overline{f_2}/\left(\og_{12}x\right)$.
Working from the expression for $f_2$ given in the proof of Lemma~\ref{rank2-gen-lt} gives
$$h=\og_{23}y^{p+1}+\frac{1}{2}\left(\og_{12}xz^p+
x^{p-1}\left(\og_{14}\delta+\frac{\og_{23}^2y^2}{\og_{12}}\right)+\og_{24}yx^p\right).$$

\begin{theorem}  If $\gamma_{12}\not=0$ and $\gamma_{13}=0$ then
$\{x,N_G(y),h, N_G(z)\}$ is a SAGBI basis for $\field[V_M]^G$.
Furthermore, there is a single relation among the generators constructed by
subducting $\og_{23}^p N_G(y)^{p+1}-h^p$.
\end{theorem}

\begin{proof}
For $\gamma_{12}\not=0$, we can choose generators for $G$ using the left $SL_2(\fp)$ action
and a basis for $V_M$ using the right $\field\ltimes\field^*$ action,  so that
$$M:=\left[\begin{array}{cc}
1&c_{12}\\
0&c_{22}
\end{array}\right]$$
and $\og_{12}=c_{22}$.
Since $\gamma_{13}=0$, we have $c_{12}^p=c_{12}$. Thus $c_{12}\in\fp$. Again using the $SL_2(\fp)$
action
to change generators, we can take
$$M:=\left[\begin{array}{cc}
1&0\\
0&c_{22}
\end{array}\right].$$
Hence $\deg(N_G(y))=p$ and $\deg(N_G(\delta))=2p$. Therefore, using
Theorem~\ref{ffthm}, we have $\field[V_M]^G[x^{-1}]=\field[x,N_G(y),N_G(\delta)][x^{-1}]$.
An explicit calculation gives $N_G(\delta)=\delta^p-\delta x^{2p-2}c_{22}^{p-1}$.
Evaluating gives $\of_{124}=-c_{22}N_G(\delta)$. Similarly ,
$\overline{f_1}=c_{22}N_G(y)$. Using the definition of $f_2$, we have
$\gamma_{12}f_{124}=2f_2x^{p-2}-f_1^2$. Thus $N_G(\delta)$
is in the algebra generated by
$\{x, h,N_G(y)\}$. Hence  $\field[x,N_G(y),h][x^{-1}]=\field[V_M]^G[x^{-1}]$.

Referring to the definition of $\widetilde{N}$ given above,
\begin{eqnarray*}
\overline{\widetilde{N}}&=&-\og_{12}^2\of_2^p
+\og_{12}\og_{23}^p \left(x^p\of_1^{p+1}-x^{2p-2}\of_2\of_1^{p-1}\right)\\
&=&-\og_{12}^{p+2}x^p h^p
+\og_{12}\og_{23}^p \left(x^p\of_1^{p+1}-x^{2p-1}\og_{12}h\of_1^{p-1}\right)\\
&=&x^p\og_{12}^{p+2}\left(\og_{23}^p N_G(y)^{p+1}-h^p\right)+\og_{12}^{p+1}\og_{23}^px^{2p-1}h N_G(y)\\
&=&x^p\og_{12}^{p+2}\left(\og_{23}^p N_G(y)^{p+1}-h^p+\og_{12}^{-1}\og_{23}^px^{2p-1}h N_G(y)\right).
\end{eqnarray*}
Therefore, using Lemma~\ref{rank2-gen-lt},  $N':=2\overline{\widetilde{N}}/\left(c_{22}x\right)^p$ has lead term
$z^{p^2}$. Define $\mathcal{B}':=\{x,N_G(y), h, N'\}$. Applying Theorem~\ref{genthm} shows that $\mathcal{B}'$
is a SAGBI basis for $\field[V_M]^G$. Thus the lead term algebra of $\field[V_M]^G$ is generated
by $\{x, y^p, y^{p+1}, z^{p^2}\}$. Therefore $\{x,N_G(y),h,N_G(z)\}$ is a SAGBI basis for
$\field[V_M]^G$. Note that while the right $\field\ltimes\field^*$ action results in a change of variables,
the lead term algebra is preserved.
\end{proof}

For a representation with $\gamma_{12}=0$ and $\gamma_{13}\not=0$, define
$d:=\overline{f_1}/x^{p-2}=\og_{13}\delta+\og_{23}yx$.

\begin{theorem}  If $\gamma_{13}\not=0$ and $\gamma_{12}=0$ then
$\{x,d,N_G(y), N_G(z)\}$ is a SAGBI basis for $\field[V_M]^G$.
Furthermore, there is a single relation among the generators constructed by
subducting $d^p-\og_{13}^2N_G(y)^2$.
\end{theorem}

\begin{proof}
For $\gamma_{12}=0$ and $\gamma_{13}\not=0$, we can choose generators for $G$ using the left $SL_2(\fp)$ action
and a basis for $V_M$ using the right $\field\ltimes\field^*$ action,  so that
$$M:=\left[\begin{array}{cc}
1&c_{12}\\
0&0
\end{array}\right].$$
With this choice, $\og_{13}=c_{12}^p-c_{12}$, $\og_{23}=0$ and $d=\og_{12}\delta$.
Since $V_M$ is a symmetric square representation, the result follows from Theorem~\ref{spthm}.
\end{proof}

\begin{remark} A  representation with $\gamma_{13}=0$ and $\gamma_{12}=0$ is not faithful.
If the image of the representation is not equal to the identity, then representation is a rank one symmetric square representation
and, using Theorem~\ref{spthm}, the  ring of invariants is a hypersurface with generators in degrees $(1,2,p,p)$ 
and a relation in degree $2p$. This is a familiar example, first computed by Dickson \cite[Lecture III \S 7]{dickson_madison}.
\end{remark}

\section{The invariants for rank three dimension three}\label{r3d3_sec}

In this section $G=(\zp)^3=\langle e_1,e_2,e_3\rangle$. We start with the generic representation over
$\rfield:=\fp(x_{11},x_{12},x_{13},x_{21},x_{22},x_{23})$ given by $e_i\mapsto \sigma(x_{1i},x_{2,i})$. Thus, we assume $p>2$.
Consider $\mathcal{B}:=\{x,f_1,f_2,N_G(z)\}$ with $f_1$ and $f_2$ defined as in Section~\ref{field_section}:
\begin{eqnarray*}
f_1&=&-\gamma_{123}\delta^p-\gamma_{124}x^py^p-\gamma_{134}\delta x^{2p-2}-\gamma_{234}yx^{2p-1},\\
f_2&=&\gamma_{123}y^{p^2}-\gamma_{125}y^px^{p^2-p}-\gamma_{135}\delta x^{p^2-2}-\gamma_{235}yx^{p^2-1}.
\end{eqnarray*}
Thus $\lt(f_1)=-\gamma_{123}y^{2p}$ and $\lt(f_2)=\gamma_{123}y^{p^2}$. There is a single
non-trivial \tat: $(f_1^p,f_2^2)$.
\begin{lemma}
$$\lt\left(f_1^p+\gamma_{123}^{p-2}f_2^2+2(-\gamma_{123})^{(p-3)/2}\gamma_{125}x^{p^2-p}f_1^{(p+1)/2}\right)
=-2\gamma_{123}^{p-1}\gamma_{135}x^{p^2-2}y^{p^2+2}.$$
\end{lemma}
\begin{proof} We work modulo the ideal in $\rfield[x,y,z]$ generated by $x^{p^2-1}$, which we
denote by $\mathfrak{m}$. Using the formulae given above, $f_1^p\equiv_{\mathfrak{m}}-\gamma_{123}^py^{2p^2}$,
$x^{p^2-p}f_1\equiv_{\mathfrak{m}}-\gamma_{123}x^{p^2-p}y^{2p}$ and
$$f_2^2\equiv_{\mathfrak{m}}\gamma_{123}^2y^{2p^2}-2\gamma_{123}\gamma_{125}x^{p^2-p}y^{p^2+p}
-2\gamma_{123}\gamma_{135}x^{p^2-2}\delta y^{p^2}.
$$
Thus
$$f_1^p+\gamma_{123}^{p-2}f_2^2\equiv_{\mathfrak{m}}-2\gamma_{123}^{p-1}y^{p^2}\left(\gamma_{125}x^{p^2-p}y^p
+\gamma_{135}x^{p^2-2}\delta\right)$$
and
$$2(-\gamma_{123})^{(p-3)/2}\gamma_{125}x^{p^2-p}f_1^{(p+1)/2}\equiv_{\mathfrak{m}}2\gamma_{123}^{p-1}x^{p^2-p}y^{p^2+p}.$$
Hence
$$f_1^p+\gamma_{123}^{p-2}f_2^2+2(-\gamma_{123})^{(p-3)/2}\gamma_{125}x^{p^2-p}f_1^{(p+1)/2}\equiv_{\mathfrak{m}}
-2\gamma_{123}^{p-1}\gamma_{135}x^{p^2-2}y^{p^2}\delta$$
and the result follows.
\end{proof}
Define
$$f_3:=\frac{f_1^p+\gamma_{123}^{p-2}f_2^2+2\left(-\gamma_{123}\right)^{(p-3)/2}\gamma_{125}x^{p^2-p}f_1^{(p+1)/2}}
{-2x^{p^2-2}}$$
and $\mathcal{B}_1:=\mathcal{B}\cup\{f_3\}$. There is a single new non-trivial \tat\ among the elements of $\mathcal{B}_1$:
$(f_3^p,f_1f_2^p)$.

\begin{lemma}\label{f3lem} For $\mathfrak{h}:=(x^3,x^2y)\rfield[x,y,z]$,
$$f_3\equiv_{\mathfrak{h}}
\gamma_{123}^{p-1}\left(\gamma_{135}\delta y^{p^2}+\gamma_{235}xy^{p^2+1}-\frac{1}{2}\gamma_{123}x^2z^{p^2}\right).$$
\end{lemma}
\begin{proof} We consider $x^{p^2-2}f_3$ modulo $\mathfrak{h}':=(x^{p^2+1},yx^{p^2})\rfield[x,y,z]$.
Using the definitions:
\begin{eqnarray*}
f_1^p&\equiv_{\mathfrak{h}'}&-\gamma_{123}^p\delta^{p^2}\equiv_{\mathfrak{h}'}-\gamma_{123}^p\left(y^{2p^2}-x^{p^2}z^{p^2}\right);\\
x^{p^2-p}f_1^{(p+1)/2}&\equiv_{\mathfrak{h}'}& x^{p^2-p}(-\gamma_{123})^{(p+1)/2}\delta^{p(p+1)/2}
\equiv_{\mathfrak{h}'} x^{p^2-p}(-\gamma_{123})^{(p+1)/2}y^{p^2+p};\\
f_2^2&\equiv_{\mathfrak{h}'}&
\gamma_{123}^2y^{2p^2}
-2\gamma_{123}x^{p^2-p}\left(\gamma_{125}y^{p^2+p}+\gamma_{135}x^{p-2}\delta y^{p^2}+\gamma_{235}x^{p-1}\delta y^{p^2+1}\right).
\end{eqnarray*}
Forming the appropriate linear combination gives
$$-2x^{p^2-2}f_3\equiv_{\mathfrak{h}'} \gamma_{123}^px^{p^2}z^{p^2}
-2\gamma_{123}^{p-1}x^{p^2-p}\left(\gamma_{135}x^{p-2}\delta y^{p^2}+\gamma_{235}x^{p-1}\delta y^{p^2+1}\right),$$
and the required description of $f_3$ modulo $\mathfrak{h}$ follows.
\end{proof}

\begin{lemma} \label{rank3-gen-lt} There exist $c_1,c_2,c_3\in\rfield$ such that
\begin{eqnarray*}
\widetilde{N}&:=&f_3^p+\gamma_{135}^p\gamma_{123}^{p^2-2p-1}f_1f_2^p-c_1x^pf_1^{(p^2+1)/2}
-c_2x^{2p-2}f_2^{p-1}f_3\\
&&-c_3x^{2p-1}f_3^{(p+1)/2}f_2^{(p-3)/2}f_1^{(p-1)/2}
\end{eqnarray*}
has lead term $-\frac{1}{2}\gamma_{123}^{p^2}x^{2p}z^{p^3}$.
\end{lemma}
\begin{proof} We work modulo the ideal in $\rfield[x,y,z]$ generated by $x^{2p+1}$ and $yx^{2p}$,
which we denote by $\mathfrak{n}$. By definition, $f_2^p\equiv_{\mathfrak{n}}\gamma_{123}^py^{p^3}$.
Thus
$$
f_3^p+\gamma_{135}^p\gamma_{123}^{p^2-2p-1}f_1f_2^p\equiv_{\mathfrak{n}}f_3^p-\gamma_{135}^p\gamma_{123}^{p^2-p-1}f_1y^{p^3}.$$
Using Lemma~\ref{f3lem} and the definition of $f_1$ gives
\begin{eqnarray*}
f_3^p+\gamma_{135}^p\gamma_{123}^{p^2-2p-1}f_1f_2^p &\equiv_{\mathfrak{n}}&
\gamma_{123}^{p^2-p-1}\left(\gamma_{123}\left(\gamma_{235}^px^py^{p^3+p}-\frac{1}{2}\gamma_{123}^px^{2p}z^{p^3}\right)\right.\\
& - &
\gamma_{135}^py^{p^3}\left(\left.\gamma_{124}x^py^p+\gamma_{134}\delta x^{2p-2}+\gamma_{234}yx^{2p-1}\right)\right)\\
&\equiv_{\mathfrak{n}}& \tilde{c_1}x^py^{p^3+p}+\tilde{c_2}x^{2p-2}\delta y^{p^3}+\tilde{c_3}x^{2p-1}y^{p^3+1}
-\frac{1}{2}\gamma_{123}^{p^2}x^{2p}z^{p^3}
\end{eqnarray*}
for some $\tilde{c_i}\in\rfield$.
Using the definitions and Lemma~\ref{f3lem}, we have
\begin{eqnarray*}
x^pf_1^{(p^2+1)/2}&\equiv_{\mathfrak{n}}& (-\gamma_{123})^{(p^2+1)/2}x^py^{p^3+p},\\
x^{2p-2}f_2^{p-1}f_3 &\equiv_{\mathfrak{n}}& \gamma_{123}^{2p-2}\left(x^{2p-2}\gamma_{135}\delta y^{p^3} +\gamma_{235}x^{2p-1}y\right),\\
{\rm and}\ x^{2p-1}f_3^{(p+1)/2}f_2^{(p-3)/2}f_1^{(p-1)/2} &\equiv_{\mathfrak{n}}& (-1)^{(p-1)/2}\beta x^{2p-1}y^{p^3+1},
\end{eqnarray*}
with $\beta$ a monomial in $\gamma_{123}$ and $\gamma_{135}$. We then solve for the required $c_i$.
\end{proof}

\begin{theorem}\label{rk3gen}
For the generic rank three representation, $\mathcal{B}=\{x,f_1,f_2,f_3,N_G(z)\}$ is a SAGBI basis for $\rfield[V_3]^G$.
Furthermore, there are two relations among the generators, one constructed by subducting
$f_3^p-\gamma_{135}^p\gamma_{123}^{p^2-2p-1}f_1f_2^p$ and the other given by
$$f_1^p+\gamma_{123}^{p-2}f_2^2+2\left(-\gamma_{123}\right)^{(p-3)/2}\gamma_{125}x^{p^2-p}f_1^{(p+1)/2}
+2x^{p^2-2}f_3=0.$$
\end{theorem}
\begin{proof} Define $N:=\widetilde{N}/x^{2p}$.
It follows from Lemma~\ref{rank3-gen-lt} and the definition of $f_3$ that
$\mathcal{B}':=\{x,f_1,f_2,,f_3,N\}$ is a SAGBI basis for the algebra it generates.
Applying Theorem~\ref{fieldthm} and Theorem~\ref{genthm} proves that $\mathcal{B}'$ is a SAGBI basis for $\rfield[V_3]^G$.
Thus the lead term algebra of $\rfield[V_3]^G$ is generated by $\{x,y^{2p},y^{p^2},y^{p^2+2}, z^{p^3}\}$.
Since $\lm(\mathcal{B})=\{x,y^{2p},y^{p^2},y^{p^2+2}, z^{p^3}\}$, $\mathcal{B}$ is also
a SAGBI basis for $\rfield[V_3]^G$. The two non-trivial \tat s subduct to produce the relations.
\end{proof}

For $\field$ an arbitrary field of characteristic $p$, consider the representation $V_M$  for
$$M:=\left[\begin{array}{ccc}
c_{11}&c_{12}&c_{13}\\
c_{21}&c_{22}&c_{23}
\end{array}\right]$$
with $c_{ij}\in\field$, i.e., $e_i\mapsto \sigma(c_{1i},c_{2i})$.
Let $\psi_M:\fp[x_{ij}]\to\field$ denote the evaluation map: $\psi_M(x_{ij})=c_{ij}$.
For $f\in\rfield[x,y,z]$, let $\of$ denote the element of $\field[x,y,z]$
constructed by minimally clearing denominators and applying $\psi_M$ to the coefficients.
Note that we can also interpret elements of $\fp[x_{ij}]$ as regular functions on $\field^{2\times 3}$,
the space of representations of the given type.

\begin{theorem} If $\gamma_{123}\not=0$ and $\gamma_{135}\not=0$, then
$\{x,\of_1,\of_2,\of_3,N_G(z)\}$ is a SAGBI basis for $\field[V_M]^G$.
Furthermore,  there are two relations among the generators, one constructed by subducting
$\of_3^p-\og_{135}^p\og_{123}^{p^2-2p-1}\of_1\of_2^p$ and the other given by
$$\of_1^p+\og_{123}^{p-2}\of_2^2+2\left(-\og_{123}\right)^{(p-3)/2}\og_{125}x^{p^2-p}\of_1^{(p+1)/2}
+2x^{p^2-2}\of_3=0.$$
\end{theorem}

\begin{proof} We show that, in essence,  the proof of Theorem~\ref{rk3gen} survives evaluation.
Referring to the proof of Lemma~\ref{rank3-gen-lt}, we see that
$\tilde{c}_i\in \fp[x_{jk}]$ and that we can solve for the $c_i$ after inverting $\gamma_{123}$
and $\gamma_{135}$. Therefore, with  $N:=\widetilde{N}/x^{2p}$, we have
$\lm(\overline{N})=z^{p^3}$.
Let $A$ denote the algebra generated by $\mathcal{B}':=\{x,\of_1, \of_2,\of_3,\overline{N}\}$.
Since $\gamma_{123}\not=0$, we have
$\lm(\of_1)=y^{2p}$ and $\lm(\of_2)=y^{p^2}$. Using Lemma~\ref{f3lem},
$\lm(\of_3)=y^{p^2+2}$. The \tat\ $(\of_1^p,\of_2^2)$ subducts to zero using the definition of $\of_3$
and the \tat\ $(\of_3^p,\of_1\of_2^p)$ subducts to zero using the definition  of $\overline{N}$.
Thus $\mathcal{B}'$ is a SAGBI basis for $A$.

Referring to the definition of $f_2$, we see that $\of_2$ has degree $1$ as a polynomial in $z$ with coefficient
$\og_{135}x^{p^2-2}$. Thus $\field[x,\of_2, N_G(y)][x^{-1}]=\field[V_M]^G[x^{-1}]$.
Using the notation of section~\ref{field_section}, observe that $N_G(y)=\og_{135}\of_{1357}$.
By definition $f_1=f_{1234}$ and $f_2=f_{1235}$. Applying Lemma~\ref{plucker} to
the subsequence $(1,4,7)$ of $(1,3,4,5,7)$ shows that $\gamma_{345}\widetilde{f}_{1357}$
can be written as an $\fp$-linear combination of
$\gamma_{135}\widetilde{f}_{3457}$ and $\gamma_{357}\widetilde{f}_{1345}$.
Observe that $\gamma_{345}=\gamma_{123}^p$ and $\widetilde{f}_{3457}=\widetilde{f}_{1235}^p $ .
Iteratively applying Lemma~\ref{plucker} as in the proof of Lemma~\ref{powers} allows us to write
$\gamma_{123}^p\widetilde{f}_{1357}$ as an element of the $\fp[x_{jk}]$-span
of $\{\widetilde{f}_{1234}^{p^i},\widetilde{f}_{1235}^{p^i}\mid i=0,1\}$.
Thus $N_G(y)$ is in the algebra generated by $\{x,\of_1,\of_2\}$.
Therefore $A[x^{-1}]=\field[V_M]^G[x^{-1}]$ and the result follows from Theorem~\ref{genthm}.
\end{proof}

For a representation with $\gamma_{123}\not=0$ and $\gamma_{135}=0$, arguing as in the
proof of Lemma~\ref{f3lem}, 
$\of_3/x$ is equivalent  to $\og_{123}^{p-1}\left(\og_{235}y^{p^2+1}-\og_{123}xz^{p^2}\right)$, modulo the ideal
$(x^2,xy)\field[x,y,z]$.
The Pl\"ucker relation for $\Gamma$  determined by the sequences $(1,3)$ and $(2,3,4,5)$ gives
$\gamma_{123}\gamma_{345}-\gamma_{134}\gamma_{235}+\gamma_{135}\gamma_{234}=0$.
Since $\gamma_{123}^p=\gamma_{345}$ and $\og_{135}=0$, we have
$\og_{134}\og_{235}=\og_{123}^{p+1}\not=0$. Hence $\lm(\of_3/x)=y^{p^2+1}$.

\begin{theorem} If $\gamma_{123}\not=0$ and $\gamma_{135}=0$,
then $\{x,\of_1,N_G(y),\of_3/x,N_G(z)\}$ is a SAGBI basis for $\field[V_M]^G$.
Furthermore, there are two relations among the generators
constructed by subducting the \tat s 
$\left((\of_3/x)^p,\of_1^{(p^2+1)/2}\right)$ and $(\of_1^p,N_G(y)^2)$.
\end{theorem}

\begin{proof} Define 
$\widetilde{N}:=\left(\of_3/\left(x\og_{123}^p\right)\right)^p
-\left(\og_{235}/\og_{123}\right)^p\left(-\of_1/\og_{123}\right)^{(p^2+1)/2}$.
Working modulo the ideal $(x^{p+1},x^py)\field[x,y,z]$, we see that
$\lt(\widetilde{N})=(-xz^{p^2}/2)^p$. Define $N:=\widetilde{N}/x^p$
and $\mathcal{B}:=\{x,\of_1,\of_2,\of_3/x,N\}$.
Let $A$ denote the algebra generated by $\mathcal{B}$.
There are two non-trivial \tat s among the elements of $\mathcal{B}$:
$(\of_1^p,\of_2^2)$ which subducts to zero using the definition of $f_3$,
and $\left((\of_3/x)^p,\of_1^{(p^2+1)/2}\right)$ which subducts to zero using the
definition of $N$. Therefore $\mathcal{B}$ is a SAGBI basis for $A$.

For $\gamma_{123}\not=0$ and $\gamma_{135}=0$, we can choose generators for $G$ and a basis for $V_M$ so that
$$M=\left[\begin{array}{ccc}
1&c_{12}&0\\
0&c_{22}&c_{23}
\end{array}\right].$$
Using this description of $V_M$, the orbit of $y$ is 
$\{y+(\alpha+c_{12}\beta)x\mid \alpha,\beta\in\fp\}$.
Calculating the orbit product gives $N_G(y)=\of_2/\gamma_{123}\in A$.   

We now show that $A[x^{-1}]=\field[V_M]^G[x^{-1}]$.
As in the proof of Theorem~\ref{ffthm}, observe that 
$\field[x,y,z][x^{-1}]=\field[x,y,\delta/x][x^{-1}]$.
Since $f_1=-\gamma_{123}\delta^p-\gamma_{124}x^py^p-\gamma_{134}\delta x^{2p-2}-\gamma_{234}yx^{2p-1}$,
we see that $x_1^{-p}\of_1\in\field[x,y,\delta/x]^G$.
Furthermore, $N_G(y)\in \field[x,y,\delta/x]^G$ and
$\deg(N_G(y))\deg(x_1^{-p}\of_1)=p^3=|G|$.
Therefore, using \cite[3.7.5]{DK}, we have $\field[x,y,\delta/x]^G=\field[x,N_G(y),x_1^{-p}\of_1]$.
Hence
$$\field[x,y,z]^G[x^{-1}]=\field[x,y,\delta/x]^G[x^{-1}]=\field[x,\of_1,\of_2][x^{-1}]=A[x^{-1}].$$

Since $A[x^{-1}]=\field[V_M]^G[x^{-1}]$ and $\field[V_M]^G$ is integral over $A$, 
it follows from Theorem~\ref{genthm} that
$\mathcal{B}$ is a SAGBI basis for $\field[V_M]^G$. Hence the lead term algebra
of $\field[V_M]^G$ is generated by $\{x,y^{2p},y^{p^2},y^{p^2+1},z^{p^3}\}$
and $\{x,\of_1,N_G(y),\of_3/x,N_G(z)\}$ is a SAGBI basis for $\field[V_M]^G$.
\end{proof}

For the rest of this section, we consider representations for which $\gamma_{123}=0$.
If $c_{1j}=0$ for all $j\in\{1,2,3\}$, then $\field[V_M]^G=\field[x,y,N_G(z)]$.
Thus we will also assume that at least one $c_{1j}$ is non-zero.
Therefore, we can choose generators for $G$ and a basis for $V_M$ so that
$$M=\left[\begin{array}{ccc}
1&c_{12}&c_{13}\\
0&c_{22}&c_{23}
\end{array}\right]$$
and $c_{23}(c_{12}^p-c_{12})=c_{22}(c_{13}^p-c_{13})$.
If $\gamma_{124}=0$, we can take
\begin{equation}\label{matrix124}
M=\left[\begin{array}{ccc}
1&c_{12}&c_{13}\\
0&0&c_{23}
\end{array}\right]
\end{equation}

and if $\gamma_{135}=0$, we can take
\begin{equation}\label{matrix135}
M=\left[\begin{array}{ccc}
1&c_{12}&0\\
0&c_{22}&c_{23}
\end{array}\right].
\end{equation}

\begin{theorem} If $V_M$ is a representation of type $(1,1,1)$ and
$\gamma_{123}=\gamma_{124}=\gamma_{135}=0$, then $V_M$ is not faithful.
\end{theorem}
\begin{proof} Since $V_M$ has type $(1,1,1)$, at least one of the $c_{1j}$ is non-zero. Therefore,
since $\gamma_{124}=\gamma_{135}=0$,
we can choose generators for $G$ and a basis for $V_M$ so that
$$M=\left[\begin{array}{ccc}
1&c_{12}&0\\
0&0&c_{23}
\end{array}\right].$$
Thus $\og_{123}=c_{23}(c_{12}-c_{12}^p)=0$. If $c_{23}=0$ then $e_3$ acts trivially.
If $c_{12}^p-c_{12}=0$ then $c_{12}\in\fp$ and $c_{12}e_1-e_2$ acts trivially.
\end{proof}

\begin{theorem}
If $\gamma_{123}=0$, $\gamma_{124}=0$ and $\gamma_{135}\not=0$, then
$V_M$ is a symmetric square representation.
\end{theorem}
\begin{proof}
Using the form for $M$ given in Equation~\ref{matrix124}, $c_{22}=0$ and
$c_{23}(c_{12}^p-c_{12})=0$. However, $\gamma_{135}\not=0$
means that $c_{12}^p-c_{12}\not=0$. Thus $c_{23}=0$ and $V_M$ is a symmetric square representation.
\end{proof}

Consider the case $\gamma_{123}=0$, $\gamma_{124}\not=0$, and $\gamma_{135}\not=0$.
The Pl\"ucker relations for $\Gamma$ determined by the sequences $(1,2),(1,3,4,5)$
and $(2,3),(1,2,4,5)$ give:
\begin{eqnarray*}
\gamma_{123}\gamma_{135}-\gamma_{124}\gamma_{135}+\gamma_{125}\gamma_{134}&=&0,\cr
\gamma_{123}\gamma_{245}+\gamma_{234}\gamma_{125}-\gamma_{235}\gamma_{124}&=&0.
\end{eqnarray*}
If $\gamma_{123}=0$, $\gamma_{124}\not=0$, and $\gamma_{135}\not=0$, then
$$f:=\of_1(-\og_{124}x^p)^{-1}=\of_2(-\og_{125}x^{p^2-p})^{-1}
=y^p+\frac{\og_{134}}{\og_{124}}\delta x^{p-2}+\frac{\og_{234}}{\og_{124}}y x^{p-1}$$
with $c:=\og_{134}/\og_{124}\not=0$.
Since $\gamma_{135}\not=0$, we have
$$N_G(y)=y^{p^3}+\frac{\og_{137}}{\og_{135}}y^{p^2}x^{p^3-p^2}+\frac{\og_{157}}{\og_{135}}y^p x^{p^3-p}+
\frac{\og_{357}}{\og_{135}}y x^{p^3-1}.$$
Note that $\og_{357}/\og_{135}=\og_{135}^{p-1}=d_3(W)$, $\og_{157}/\og_{135}=d_2(W)$ and
$\og_{137}/\og_{135}=d_1(W)$, where $d_i(W)$ is the $i^{th}$ Dickson invariant for 
$W={\rm Span}_{\fp}\{c_{11},c_{12},c_{13}\}$. 
Working modulo the ideal $(x^{p^3-p^2-p-1})\field[V_M]$,
we see that the lead term of
$$\widetilde{h}:=N_G(y)-f^{p^2}+c^{p^2}f^{2p}x^{p^3-2p^2}
-2c^{p^2+p}f^{p+2}x^{p^3-p^2-2p}$$
is $-4c^{p^2+p+1}y^{p^2+p+2}x^{p^3-p^2-p-2}$. Thus $h:=\widetilde{h}(-4c^{p^2+p+1}x^{p^3-p^2-p-2})^{-1}$
has lead term $y^{p^2+p+2}$.

\begin{theorem}
If $\gamma_{123}=0$, $\gamma_{124}\not=0$ and $\gamma_{135}\not=0$, then
$\{x, f,h,N_G(z)\}$ is a SAGBI basis for $\field[V_M]^G$. Furthermore,
there is a single relation among the generators constructed by subducting
$f^{p^2+p+2}-h^p$.
\end{theorem}
\begin{proof} 
For $p>3$ define
\begin{eqnarray*}
\widetilde{N}&:=&f^{p^2+p+2}-h^p
-2cf^{p^2}hx^{p-2}
-\frac{d_1(W)}{2c^{p^2+p}}f^{p^2+p}x^{2p}\\&&
+\frac{d_1(W)}{2c^{p^2}}f^{p^2+2}x^{p^2}
-\frac{d_1(W)}{c^{p^2-1}}f^{p^2-p}hx^{p^2+p-2}
-\frac{d_2(W)}{2c^{p^2+p}}f^{p^2+1}x^{p^2+p}\\&&
+\frac{d_2(W)}{2c^{p^2+p-1}}hf^{p^2-p-1}x^{p^2+2p-2}
-\frac{2d_3(W)}{c^{p^2+p}}f^{(p-3)(p+1)/2}x^{p^2+2p-1}\left(\frac{h}{4}\right)^{(p+1)/2}
\end{eqnarray*}
and for $p=3$ define
\begin{eqnarray*}
\widetilde{N}&:=&f^{p^2+p+2}-h^p
+cf^{p^2}hx^{p-2}
+\left(\frac{d_1(W)+c^{18}}{c^{12}}\right)f^{p^2+p}x^{2p}\\&&
-\left(\frac{d_1(W)+c^{18}}{c^{9}}\right)f^{p^2+2}x^{p^2}
-\left(\frac{d_1(W)+c^{18}}{c^{8}}\right)f^{p^2-p}hx^{p^2+p-2}\\&&
+\left(\frac{d_2(W)+d_1(W)c^6+c^{24}}{c^{12}}\right)f^{p^2+1}x^{p^2+p}\\&&
-\left(\frac{d_2(W)+d_1(W)c^6+c^{24}}{c^{11}}\right)hf^{p^2-p-1}x^{p^2+2p-2}\\&&
+\left(\frac{d_3(W)+d_1(W)c^8+d_2(W)c^2+c^{26}}{c^{12}}\right)x^{14}h^2.
\end{eqnarray*}
Working modulo the ideal $(x^{p^2+2p+1},yx^{p^2+2p})\field[x,y,z]$, an explicit calculation gives
$\lt(\widetilde{N})=\frac{1}{4}z^{p^3}x^{p^2+2p}/c^{p^2+p}$. Define $N:=\widetilde{N}/x^{p^2+p}$ and
$\mathcal{B}':=\{x,f,h,N\}$. Let $A$ denote the subalgebra generated by $\mathcal{B}'$ and note that
$\mathcal{B}'$ is a SAGBI basis for $A$.
It follows from the definition of $\widetilde{h}$ that $N_G(y)\in A$. Observe that $f$ is degree $1$
in $z$ with coefficient $-cx^{p-1}$. Therefore, by \cite[Theorem~2.4]{CampChuai}, $A[x^{-1}]=\field[V_M]^G$.
Applying Theorem~\ref{genthm}, we see that $\mathcal{B}'$ is a SAGBI basis for $\field[V_M]^G$. Hence
the lead term algebra of $\field[V_M]^G$ is generated by $\{x,y^p,y^{p^2+p+2},z^{p^3}\}$. Therefore
$\{x,f,h,N_G(z)\}$ is a SAGBI basis for $\field[V_M]^G$.
\end{proof}

Consider the case $\gamma_{123}=\gamma_{135}=0$ and $\gamma_{124}\not=0$. Since $\gamma_{123}=0$, using the form
for $M$ given in Equation~\ref{matrix135}, we see that $c_{23}(c_{12}^p-c_{12})=0$. If we assume $V_M$ is faithful,
then $c_{23}\not=0$ and $c_{12}\in\fp$. In this case $N_G(y)=y^p-yx^{p-1}$ and 
$$N_G(\delta)=\delta^{p^2}+d_1(U)\delta^p x^{2p^2-2p}+d_2(U)\delta x^{2p^2-2}$$
where $d_1(U)=\og_{126}/\og_{124}$ and $d_2(U)=\og_{146}/\og_{124}$ are the
Dickson invariants for $U={\rm Span}_{\fp}\{c_{22},c_{23}\}$. Working modulo the ideal
$(x^{p^2})\field[V_M]$ we see that the lead term of
$$\widetilde{g}:=N_G(\delta)-N_G(y)^{2p}-2N_G(y)^{p+1}x^{p(p-1)}$$
is $-2y^{p^2+1}x^{p^2-1}$. Thus $g:=\widetilde{g}(-2x^{p^2-1})^{-1}$ has lead term $y^{p^2+1}$.

\begin{theorem}
If $V_M$ is a faithful representation with 
$\gamma_{123}=\gamma_{135}=0$ and $\gamma_{124}\not=0$, then $\{x,N_G(y),g,N_G(z)\}$ is a
SAGBI basis for $\field[V_M]^G$. Furthermore,
there is a single relation among the generators constructed by subducting
$N_G(y)^{p^2+1}-g^p$.
\end{theorem}
\begin{proof}
Define
$\widetilde{N}:=N_G(y)^{p^2+1}-g^p+gN_G(y)^{p(p-1)}x^{p-1}$.
Working modulo the ideal $(x^{p+1},yx^{p})\field[x,y,z]$, an explicit calculation gives
$\lt(\widetilde{N})=\frac{1}{2}z^{p^3}x^p$. Define $N:=\widetilde{N}/x^p$ and
$\mathcal{B}':=\{x,N_G(y),g,N\}$. Let $A$ denote the subalgebra generated by $\mathcal{B}'$ and note that
$\mathcal{B}'$ is a SAGBI basis for $A$.
It follows from the definition of $\widetilde{g}$ that $N_G(\delta)\in A$.
Since $\deg(N_G(y))\deg(N_G(\delta))=2p^3=2|G|$,
applying Theorem~\ref{ffthm}, we see that $\field[V_M]^G[x^{-1}]=\field[x,N_G(y),N_G(\delta)][x^{-1}]=A[x^{-1}]$.
It then follows from Theorem~\ref{genthm} that $\mathcal{B}'$ is a SAGBI basis for $\field[V_M]^G$. Hence
the lead term algebra of $\field[V_M]^G$ is generated by $\{x,y^p,y^{p^2+1},z^{p^3}\}$. Therefore
$\{x,N_G(y),g,N_G(z)\}$ is a SAGBI basis for $\field[V_M]^G$.
\end{proof}

\section{Conjectures and Conclusions}\label{con_sec}

Numerous computer calculations using Magma~\cite{magma} support the following.
\begin{conjecture}
Let $V$ denote the generic three dimensional representation of $G=(\zp)^r=\langle e_1,\ldots,e_r\rangle$ over
$\rfield=\fp(x_{1j},x_{2j}\mid j=1,\ldots,r)$ and define $s=\lceil \frac{r}{2}\rceil$. Then $\rfield[V]^G$ is
a complete intersection with embedding dimension $s+3$. Furthermore, there exists a SAGBI  basis
$\{x,f_1,\ldots,f_{s+1},N_G(z)\}$ such that:\\
(a) if $r=2s$ then $\lm(f_1)=y^{p^s}$, $\lm(f_i)=y^{p^{s+i-2}+2p^{s-i+1}}$ for $i>1$ and the
relations come from subducting the \tat s $(f_2^p,f_1^{p+2})$ and
$\left(f_i^p,f_{i-1}f_1^{(p^2-1)p^{i-3}}\right)$ for $i>2$;\\
(b) if $r=2s-1$ then $\lm(f_1)=y^{2p^{s-1}}$, $\lm(f_2)=y^{p^s}$
and $\lm(f_i)=y^{p^{s+i-3}+2p^{s-i+1}}$ for $i>2$ and the
relations are constructed by subducting the \tat s $(f_1^p,f_2^2)$,
$(f_3^p,f_1f_2^p)$ and
$\left(f_i^p,f_{i-1}f_2^{(p^2-1)p^{i-4}}\right)$ for $i>3$.
\end{conjecture}

In Section~\ref{cl3d_sec}, we showed that the only three dimensional representations
of $G=(\zp)^r$ for which the ring of invariants fails to be a polynomial algebra,
are of the form $V_M$ for some $M\in\field^{2\times r}$. Rings of invariants for
these representations are parameterised by $\field^{2\times r}\rightquot GL_r(\fp)$.
In Section~\ref{r2d3_sec}, we identified 
$\gamma_{12},\gamma_{13}\in \field[\field^{2\times 2}]^{SL_2(\fp)}$
such that representations satisfying $\gamma_{12}\not=0$ and $\gamma_{13}\not=0$
are essentially generic. In particular, these representations have rings of invariants which are 
generated in degrees $1,p,p+2,p^2$ with a single relation in degree $p(p+2)$.
The remaining rank $2$,  dimension $3$ representations fall into three cases:
$\gamma_{12}\not=0$ and $\gamma_{13}=0$, giving generators in degrees $1,p,p+1,p^2$ with a relation in degree
$p(p+1)$; $\gamma_{12}=0$ and $\gamma_{13}\not=0$, giving generators in degrees $1,2,p,p^2$ with a relation in degree $2p$;
$\gamma_{12}=0$ and $\gamma_{13}=0$, which fails to be faithful. 
In Section~\ref{r3d3_sec}, we used 
$\gamma_{123},\gamma_{135},\gamma_{124}\in \field[\field^{2\times 3}]^{SL_3(\fp)}$ to stratify 
$\field^{2\times 3}\rightquot GL_3(\fp)$.
Representations satisfying $\gamma_{123}\not=0$ and $\gamma_{135}\not=0$ are essentially generic.
Representations satisfying $\gamma_{123}\not=0$ and $\gamma_{135}=0$ form a second family
and representations satisfying $\gamma_{123}=0$ fall into four cases:
$\gamma_{135}\not=0$ and $\gamma_{124}\not=0$;
$\gamma_{135}=0$ and $\gamma_{124}\not=0$;
$\gamma_{135}\not=0$ and $\gamma_{124}=0$;
$\gamma_{135}=0$ and $\gamma_{124}=0$.
The representations with $\gamma_{123}\not=0$ have rings of invariants which are complete intersections with embedding 
dimension $5$. The rings of invariants for representations with $\gamma_{123}=0$ are hypersurfaces.
We believe that for an arbitrary $r$ it is possible to stratify $ \field^{2\times r}\rightquot GL_r(\fp)$
using elements of $\field[\field^{2\times r}]^{SL_r(\fp)}$, and that  in the essentially generic case, 
the ring of invariants is a complete intersection with embedding dimension $\lceil r/2 \rceil+3$,
and that in all other cases, the ring of invariants is a complete intersection with embedding dimension
at most $\lceil r/2 \rceil+3$.

\begin{conjecture}
For a faithful modular three dimensional representation of $(\zp)^r$, the ring of invariants is a complete
intersection with embedding dimension less than or equal to $\lceil r/2 \rceil+3$.
\end{conjecture}

\ifx\undefined\bysame
\newcommand{\bysame}{\leavevmode\hbox to3em{\hrulefill}\,}
\fi

\end{document}